\documentclass[12pt]{article}
\usepackage{a4wide}
\usepackage{amsmath,amsthm}
\usepackage{amsfonts}
\usepackage{amssymb}
\usepackage{url}
\usepackage{xspace}
\usepackage{graphicx, epstopdf}
\usepackage{pdfsync}
\usepackage{pdflscape}
\usepackage{mathrsfs}
\usepackage{framed}
\usepackage[x11names,usenames,dvipsnames,svgnames,table]{xcolor}
\usepackage{listings}
\usepackage[hidelinks]{hyperref}

\numberwithin{equation}{section}

\newtheorem{theorem}{Theorem}
\newtheorem{lemma}{Lemma}
\newtheorem{proposition}{Proposition}
\newtheorem{corollary}{Corollary}

\theoremstyle{definition}

\newtheorem{definition}{Definition}[section]

\theoremstyle{remark}
\newtheorem{remark}{Remark}[section]

\DeclareMathOperator{\linspan}{span}

\DeclareMathOperator{\dom}{dom}

\newcommand{\wT}{\hyperref[D:generalfamily]{w_\mathcal{T}}}
\newcommand{\WTC}{\hyperref[D:generalfamily]{W_\mathcal{T}^C}}
\newcommand{\vz}{\zeta}
\newcommand{\vx}[1]{\hyperref[D:family]{x^\zeta_{#1}}}
\newcommand{\vg}[1]{\hyperref[D:family]{g^\zeta_{#1}}}
\newcommand{\vf}[1]{\hyperref[D:family]{f^\zeta_{#1}}}
\newcommand{\vW}{\hyperref[D:family]{W_\zeta}}
\newcommand{\vw}{\hyperref[D:family]{w_\zeta}}
\newcommand{\sz}{\zeta^*}
\newcommand{\sW}{\hyperref[D:family]{W_{\sz}}}
\newcommand{\sx}[1]{\hyperref[D:family]{x^{\sz}_{#1}}}
\newcommand{\sg}[1]{\hyperref[D:family]{g^{\sz}_{#1}}}
\newcommand{\sff}[1]{\hyperref[D:family]{f^{\sz}_{#1}}}

\newcommand{\ee}[1]{\hyperref[P:notations]{\mathbf{u}_{#1}}}

\newcommand{\oracle}{\mathcal{O}}
\newcommand{\algorithm}{A}
\newcommand{\oraclea}{\mathcal{O}_{w_\algorithm}}
\newcommand{\risk}[2]{\mathscr{R}_{#1}{(#2)}}
\newcommand{\finstances}{\hyperref[P:smooth]{\mathcal{F}_{L, R}(\real^d)}}
\newcommand{\nonsmoothclass}{\hyperref[P:nonsmooth]{\mathcal{C}_{M, R}(\real^d)}}

\newcommand{\ax}{\xi}
\newcommand{\tmpvec}{v}

\newcommand{\real}{\mathbb{R}}
\newcommand{\realkn}{\hyperref[P:realkn]{\real^{k,N+1}}}

\title{The exact information-based complexity of 
smooth convex minimization}
\author{Yoel Drori\footnote{Google Inc. (e-mail: \texttt{dyoel@google.com})}}
\date{\today}

\begin{document}
\maketitle

\begin{abstract}
We obtain a new lower bound on the information-based complexity of first-order minimization of smooth and convex functions.
We show that the bound matches the worst-case performance of the
recently introduced Optimized Gradient Method,
thereby establishing that 
the bound is tight and 
can be realized by an efficient algorithm.
The proof is based on a novel construction technique of smooth and convex functions.

\paragraph{Keywords} Convex optimization; Complexity; Rate of convergence; Information-based complexity
\end{abstract}

\section{Introduction}
The problem of smooth and convex minimization plays a key role in a various range of applications, including
signal and image processing, communications, machine learning, and many more.
Some of the most successful approaches for solving these problems are first-order methods,
i.e., algorithms that are
only allowed to gain information on the objective by evaluating its value and gradient.
The benefit of limiting the amount of accessible information
is that these algorithms generally involve very cheap and simple computational iterations,
making them suitable for tackling large scale problems.
This benefit, however, comes with a price:
first-order methods, in general, require considerable computational effort in order reach highly accurate solutions,
making them practical when only moderate accuracy is sufficient.

As the scale of modern problems grows and finding
efficient algorithms becomes increasingly important,
a natural question that arises, and will be the main focus of this paper,
is finding the level of accuracy that can be attained by first-order methods using a bounded computational effort.
Note that there is some difficulty in answering this question that originates from the fact that
the computational effort of a first-order method consists of two parts: 
the effort in choosing the points where the objective is to be evaluated (called the \emph{search points})
and the effort in calculating the objective value and gradient at these points.
Observing that the evaluation of the objective and its gradient often dominates the computational effort of the computation
and following the theory of information-based complexity introduced in~\cite{nemirovsky1992information},
we resolve this issue by measuring
the computational effort of an algorithm by the number of times it evaluated the objective and its gradient,
neglecting the effort required for choosing the search points.



To put these concepts in more precise terms,
consider the following
unconstrained problem
\[
	(P) \quad f^* = \min_{x\in \real^d} f(x),
\]
where $f$ is a smooth and convex function.
A \emph{first-order} optimization method is an iterative algorithm that 
approximates the solution of $(P)$,
where it is only allowed to gain information on the objective $f$ via a first-order oracle, $\oracle_f$,
that is, a subroutine which given a
point in $\real^d$, returns the value of the objective and its gradient at that point.
In addition, since the problem of unconstrained minimization is invariant under translations, we also assume that the algorithm is provided with
a reference (or starting) point $x_0\in\dom(f)$ that is often assumed to be ``not too far'' from an optimal solution.
We call the pair  $(\oracle_f, x_0)$ a \emph{problem instance}, and 
for a first-order method $A$ we denote
the approximate solution generated by algorithm when applied on this problem instance by $A(\oracle_f, x_0)$%
\footnote{In order to simplify the presentation, we assume $A$ has a built-in stopping criterion that does not depend on external input.}.

Within the setting considered above, a commonly used criterion 
for measuring the accuracy of an approximate solution is
the \emph{absolute inaccuracy} criterion,
which quantifies the accuracy of an approximate solution $\xi$ for a problem instance $(\oracle_f, x_0)$ by the value of $f(\xi)-f^*$
(for alternative criteria see e.g.,~\cite[Section~3.3]{nemirovski1999optimization}).
Under this criterion, 
the \emph{efficiency estimate} of a first-order method $A$ over some given set of problem instances $\mathcal{I}$ is defined as the worst-case 
absolute inaccuracy of $A$, i.e.,
\[
	\varepsilon(A; \mathcal{I}):=\sup_{(\oracle_f,x_0)\in \mathcal{I}}f(A(\oracle_f,x_0))-f^*.
\]
We can now put the main concept addressed in this paper in formal terms:
denoting by $\mathcal{A}_N$ the set of all first-order methods that perform at most $N\in \mathbb{N}$ calls to their first-order oracle,
the \emph{minimax risk  function}~\cite{Guzman20151} associated with $\mathcal{I}$ is defined
as the infimal efficiency estimate that a first-order method can attain over $\mathcal{I}$ as a function of the computational effort $N$:
\[
	\risk{\mathcal{I}}{N} := \inf_{A\in\mathcal{A}_N} \varepsilon(A; \mathcal{I}).
\]

Note that the classical notion of \emph{information-based complexity} of the set $\mathcal{I}$ 
can be identified as the inverse to the risk function,
\[
	\mathscr{C}_{\mathcal{I}}(\varepsilon) := \min \{ N : \risk{\mathcal{I}}{N}\leq \varepsilon\},
\]
i.e., the minimal computational effort needed by a first-order method in order to reach a given worst-case accuracy level,
however,  in the following we express our results using the minimax risk function as
it proves to be more convenient.

Clearly, in order to establish an upper bound on the minimax risk of a class
it is sufficient to find 
an upper bound on the efficiency estimate of a single first-order method
(the main problem here being the identification of a good algorithm).
On the other hand, establishing lower bounds on the minimax risk
requires a more involved analysis, as the bound needs to hold for any first-order method.
Several approaches appear in the literature for establishing lower-bounds, including
resisting oracles~\cite{nemi-yudi-book83},
construction of a ``worst-case'' function~\cite{Guzman20151, nest-book-04},
and reduction to statistical problems~\cite{agarwal2012information, raginsky2011information, shapiro2005complexity}.

Note that existing works on information-based complexity
focus mainly on establishing order of magnitude bounds, 
where less attention is paid to absolute constants. Nevertheless,
the exact minimax risk was established for several important classes of problem, some of which are detailed below. 


\paragraph{Exact minimax risk results}
Consider the problem of convex quadratic minimization:
\begin{align*}
	\mathcal{P}_{\rho, R}(\real^d) := \{ (\oracle_f, x_0) : & f(x)\equiv\frac{1}{2}x^T A x + b^T x + c, \text{ for }x\in\real^d,\ A\succeq 0,\  \|A\|\leq \rho,\\ & \|x_0-x_*\|\leq R\text{ for some } x_*\in X_*(f)\}.
\end{align*}
It was established by Nemirovsky in~\cite[\S2.3.B]{nemirovsky1992information} that for $d \geq 2N+3$
exact minimax risk associated with this class  is given by
\begin{equation}\label{E:lowercomplexityquadratic}
	\risk{\mathcal{P}_{\rho, R}(\real^d)}{N} = \frac{\rho R^2}{2(2N+1)^2}.
\end{equation}
Nemirovsky also shows in~\cite{nemirovsky1992information} that this bound is attained by the Tschebyshev Methods,
and in a later work, that this bound is attained by the well-known and efficient Conjugate Gradient method
(this is a somewhat forgotten result, see (5.4.22) in \cite{nemirovski1999optimization},
where unlike classical bounds on the Conjugate Gradient method, 
this bound does not depend on nontrivial spectral properties of $A$).

Another fundamental class of problems for which the exact minimax risk result is known
is the class of non-smooth convex functions,
\phantomsection\label{P:nonsmooth}
\begin{equation*}
\begin{aligned}
	\mathcal{C}_{M, R}(\real^d) := \{ (\oracle_f, x_0) : &\text{ $f$ is a convex function in $C^{0,0}_M(\real^d)$},\\ &\|x_0-x_*\|\leq R\text{ for some } x_*\in X_*(f)\}.
\end{aligned}
\end{equation*}
For this class, the exact minimax risk was recently established in \cite{drori2014optimal},
where it was shown that for $d\geq 2N+1$\footnote{This bound can be improved to cover the case $d\geq N+1$, see Remark~\ref{R:nonsmoothex}.}
\begin{equation}\label{E:nonsmoothbound}
	\risk{\nonsmoothclass}{N} = \frac{M R}{\sqrt{N+1}}.
\end{equation}
Note that this bound is obtained by several efficient methods, including the subgradient method~\cite[\S3.2.3]{nest-book-04}
and also a family of methods recently studied in~\cite{drori2014optimal}, which are similar to Kelley's well-known Cutting-Plane Method.

In this paper we focus on the class of smooth convex functions with Lipschitz-continuous gradient:
\phantomsection\label{P:smooth}
\begin{align*}
	\finstances := \{ (\oracle_f, x_0) : &f \text{ is a convex function in $C^{1,1}_L(\real^d)$},\\ & \| x_* - x_0\|\leq R,\text{ for some } x_*\in X_*(f)\}.
\end{align*}
Since any convex quadratic function is also a smooth function with Lipschitz-continuous gradient
with constant $L=\|A\|$,
we have $\mathcal{P}_{L, R}(\real^d) \subset \finstances$ and thus we get that
the bound~\eqref{E:lowercomplexityquadratic} also forms a lower bound on the minimax risk of the more general class
of convex functions with Lipschitz-continuous gradients.
To the best of our knowledge, this bound is currently the best known lower bound on this class of problems.


Concerning an upper bound on $\risk{\finstances}{N}$,
a new method, called the  Optimized Gradient Method,  was recently introduced in \cite{drori2013performance, kim2015optimized}. 
It was shown (first numerically in~\cite{drori2013performance}, then analytically~\cite{kim2015optimized}) that
a sequence $x_0,\dots,x_N$  generated
by this method for some $N\in \mathbb{N}$, satisfies~\cite[Theorem~2]{kim2015optimized}:
\[
  f(x_N)-f(x_*) \leq \frac{L\|x_0-x_*\|^2}{2\theta_N^2}, \quad \forall x_* \in X_*(f),
\]
where
\begin{equation}\label{E:thetadef}
\begin{aligned}
	& \theta_0 = 1, \\
	& \theta_i = \frac{1+\sqrt{1 + 4 \theta_{i-1}^2}}{2}, \quad i=1,\dots,N-1, \\
	& \theta_N =\frac{1+\sqrt{1 + 8 \theta_{N-1}^2}}{2},
\end{aligned}
\end{equation}
which establishes the following upper bound on the minimax risk of smooth minimization:
\begin{equation}\label{E:upperboundonrisk}
	\risk{\finstances}{N} \leq \frac{LR^2}{2\theta_N^2}.
\end{equation}

Note that since $\sqrt{1 + 4 \theta_{i-1}^2} = 2\theta_{i-1} + o(1)$, then for $i<N$ we have
$\theta_i = i/2 + o(i)$ 
and
\[
     \theta_N = \frac{\sqrt{2}}{2} N + o(N),
\]
hence there exists a gap of about a factor of eight between the lower bound \eqref{E:lowercomplexityquadratic} and the upper bound~\eqref{E:upperboundonrisk}.
The main goal of this paper is to close this gap by showing that the bound~\eqref{E:upperboundonrisk} is in fact tight, 
i.e., the inequality \eqref{E:upperboundonrisk} can be turned into an equality.

The approach taken by this paper is motivated by the ``worst-case function'' proof technique introduced in \cite{nest-book-04}
and the construction of smooth and convex functions developed in \cite{Taylor2016}.

\paragraph{Overview of the paper}
The rest of this paper is organized as follows.
We begin, in Section~\ref{S:interpolation}, by 
introducing a novel construction of smooth and convex functions 
that satisfy a given set of requirements on their value and gradient.
We then use this construction, in Section~\ref{S:family},
to define a function
that possesses properties making it suitable 
for constructing lower bounds on the minimax risk function.
Building on these results, in Section~\ref{S:lower}, 
we establish the minimax risk associated with $\finstances$.
Finally, in Section~\ref{S:concluding} give some concluding remarks.

\paragraph{Notations}\phantomsection\label{P:notations}
We denote by $\ee{i}:=e_{i+1}$ the canonical unit  vectors
with a \emph{zero-based} index
(e.g., the canonical basis vectors for $\real^n$ are denoted by $\ee{0},\dots,\ee{n-1}$).
For a convex function $f$, we denote by $X_*(f)$
its set minimizers, which we assume to be nonempty.




\section{A smooth convex interpolation scheme}\label{S:interpolation}
In this section, we describe a general construction of smooth and convex functions that 
satisfy a set of first-order requirements.
This construction can be viewed as a generalized primal form of the interpolation scheme developed in \cite{Taylor2016}.

More precisely, given a finite index set $I$ and a set of triples $\mathcal{T}=\{(x_i, g_i, f_i)\}_{i\in I}$
with $x_i \in \real^d$, $g_i\in \real^d$, and $f_i\in \real$ for some $d\in \mathbb{N}$,
we proceed to define a convex function that has a Lipschitz-continuous gradient with constant $L> 0$
and, under certain natural conditions, 
is \emph{interpolating} through the set $\mathcal{T}$, i.e., for all $i\in I$
the value of the function at $x_i$ is $f_i$ and its gradient at this point is $g_i$.

\begin{definition}\label{D:generalfamily}
Let  $L>0$, let $\mathcal{T}$ be as defined above, and denote by
$\wT: \real^d\times \real^d \times \real^{I}\rightarrow \real$ the following convex quadratic function:
\begin{equation}
\begin{aligned}
	\wT(y,\nu,\alpha):= \frac{L}{2} \|y+\nu-\sum_{i\in I} \alpha_i  (x_i-\frac{1}{L}g_i) \|^2 +\sum_{i\in I} \alpha_i (f_i -\frac{1}{2L} \| g_i\|^2).
\end{aligned}	
\end{equation}
Then for any closed convex set $C\subset \mathbb{R}^d$ such that $0\in C$,
we define the \emph{primal interpolating function of $\mathcal{T}$ with kernel $C$},
$\WTC(y):\real^d\rightarrow \real$, by
\begin{equation}\label{E:defbigw}
\begin{aligned}
	\WTC(y):=\min_{\nu \in C,\ \alpha\in \Delta_{I}}\ \{& \wT(y,\nu,\alpha)\},
\end{aligned}	
\end{equation}
where $\Delta_I:=\{ \alpha \in \real^{I}: \sum_{i\in I} \alpha_i=1,\ \alpha_i\geq 0,\ \forall i\in I\}$ is an $|I|$-dimensional unit simplex.
\end{definition}

Note that the extra degree of freedom granted by the inclusion of the set $C$, although not necessary
for the purpose of finding an interpolating function, will be crucial for establishing the properties of the worst-case
function in the next section.

For later reference, we now state some immediate necessary and sufficient optimality conditions for the minimization problem $\WTC(y)$.
\begin{lemma}\label{L:firstorder}
Fix any $y\in\real^d$ and
suppose $(\nu^*, \alpha^*)$ is feasible for the convex optimization problem $\WTC(y)$. Then
$(\nu^*, \alpha^*)$ is  optimal for $\WTC(y)$ if and only if the following conditions hold:
\begin{equation}\label{E:nuoptcond}
	\nu^* = P_C(-y+\sum_{i\in I} \alpha^*_i  (x_i-\frac{1}{L}g_i)),
\end{equation}
and for any $j,k\in I$ such that $\alpha^*_j>0$
\begin{equation}\label{E:forderoptalpha}
\begin{aligned}
& -L \langle x_j-\frac{1}{L}g_j, y+\nu^\ast-\sum_{i\in I} \alpha_i^\ast  (x_i-\frac{1}{L}g_i)\rangle +(f_j -\frac{1}{2L} \| g_j\|^2) \\ 
&\qquad \leq  
-L \langle x_k-\frac{1}{L}g_k, y+\nu^\ast-\sum_{i\in I} \alpha_i ^\ast (x_i-\frac{1}{L}g_i)\rangle +(f_k -\frac{1}{2L} \| g_k\|^2).
\end{aligned}
\end{equation}
Here $P_C$ denotes the projection onto the set $C$.
\end{lemma}
\begin{proof}
Condition~\eqref{E:nuoptcond} follows directly from the definition of the projection function, and condition~\eqref{E:forderoptalpha}
is the first-order optimality condition for $\alpha$. See for example~\cite{bertsekas1999nonlinear}.
\end{proof}

The main property of $\WTC(y)$ is summarized by the following theorem.
\begin{theorem}\label{T:basicprop}
The function $\WTC(y)$ is convex and in $C^{1,1}_L$.
Furthermore, $\WTC(x_i)=f_i$ and $\nabla \WTC(x_i)=g_i$ for any $i\in I$ that satisfies
\begin{equation}\label{E:conic}
	\textstyle P_C(-\frac{1}{L}g_i) = 0,
\end{equation}
and 
\begin{equation}\label{E:Lipchitzxi}
	\frac{1}{2L} \| g_i- g_j\|^2 \leq f_j-f_i -\langle g_i, x_j-x_i\rangle, \quad \forall j \in I.
\end{equation}
\end{theorem}

\begin{proof}
\emph{Convexity.}
This follows from a well-known property of the infimum operator. See e.g., \cite[Proposition~2.22]{rockafellar2009variational}.

\emph{Lipschitz-continuity.}
We first show that $\WTC$ is differentiable.
Let $y_0\in \real^d$ and suppose $(\nu^{(0)}, \alpha^{(0)})$ is an optimal solution to the optimization problem $\WTC(y_0)$,
then it follows directly from the definition of the subdifferential and the definition of $\WTC$ that 
\[
	\partial \WTC(y) \subseteq \partial_y \wT(y, \nu^{(0)},\alpha^{(0)}) = \{ \nabla_y \wT(y, \nu^{(0)},\alpha^{(0)}) \}.
\]
Since $\WTC(y)$ is convex and defined over its entire domain,  $\partial \WTC(y)$ is nonempty and we get that $\partial \WTC(y) = \{ \nabla_y \wT(y, \nu^{(0)},\alpha^{(0)}) \}$,
i.e., $\WTC(y)$ is differentiable.

In order to show that $\WTC$ a has Lipschitz-continuous gradient, suppose $y_1,y_2\in \real^d$
and let $(\nu^{(1)}, \alpha^{(1)})$ be an optimal solution to the optimization problem $\WTC(y_1)$.
Since $\wT(y, \nu^{(1)},\alpha^{(1)})$ has a Lipschitz-continuous gradient with constant $L$ with respect to the variable $y$, we get	
\begin{align*}
	\WTC(y_2) 
	&\leq \wT(y_2, \nu^{(1)},\alpha^{(1)})  \\
	& \leq \wT(y_1, \nu^{(1)},\alpha^{(1)})+ \langle \nabla_y \wT(y_1, \nu^{(1)},\alpha^{(1)}), y_2-y_1\rangle +\frac{L}{2} \|y_1-y_2\|^2 \\
	& = \WTC(y_1)+\langle \nabla \WTC(y_1), y_2-y_1\rangle +\frac{L}{2} \|y_1-y_2\|^2,
\end{align*}
which is desired inequality.

\emph{Interpolation property.}
In order to establish this part of the claim, we first show
that when \eqref{E:conic} and \eqref{E:Lipchitzxi} hold,
then for all $i\in I$
an optimal solution to  the optimization problem $\WTC(x_i)$ is given by
\begin{align*}
	& \nu^*=0, \\
	& \alpha^*_j = \begin{cases}
	1, & j=i, \\
	0, & \text{otherwise}, \\
	\end{cases} \quad \forall j\in I.
\end{align*}
Indeed, substituting these values in the optimality conditions,
we get that \eqref{E:nuoptcond} follows from \eqref{E:conic}, since
\[
	P_C(-x_i+\sum_{i\in I} \alpha^*_i  (x_i-\frac{1}{L}g_i))= P_C(-x_i+(x_i-\frac{1}{L}g_i))=P_C(-\frac{1}{L}g_i) = 0,
\]
and \eqref{E:forderoptalpha} becomes
\begin{align*}
	& -L\langle x_i-\frac{1}{L}g_i, x_i-(x_i-\frac{1}{L}g_i)\rangle+ f_i-\frac{1}{2L}\|g_i\|^2 \leq \\&\qquad -L\langle x_j-\frac{1}{L}g_j, x_i-(x_i-\frac{1}{L}g_i)\rangle+ f_j-\frac{1}{2L}\|g_j\|^2 ,
\end{align*}
which reduces exactly to \eqref{E:Lipchitzxi}.

The value and gradient of $\WTC(x_i)$ can now be determined from
\[
 \WTC(x_i)=\wT(x_i, \nu^*, \alpha^*) = \frac{L}{2}\|x_i - (x_i-\frac{1}{L}g_i)\|^2+f_i-\frac{1}{2L}\|g_i\|^2=f_i
\]
and
\[
	\nabla \WTC(x_i) = \nabla_y \wT(x_i, \nu^*, \alpha^*) = L(x_i - (x_i-\frac{1}{L}g_i)) = g_i.
\]
\end{proof}

\begin{remark}
In this paper we do not consider the strongly-convex case, however, note that the construction above can be
used to generate strongly-convex functions
using the well-known property of strongly convex functions which states that a function $f$ is strongly convex with a constant $\mu$ and a has Lipschitz-continuous gradient
with constant $L$ if and only if $f(x)-\frac{\mu}{2}\|x\|^2$ is convex and 
a has Lipschitz-continuous gradient
with constant $L-\mu$.
\end{remark}

\section{A worst-case function}\label{S:family}
In this section, we construct a smooth and convex function 
and investigate its properties.
This function will form the basis to the lower complexity proof in the next section.

\begin{definition}\label{D:family}
Let $N\in \mathbb{N}$ and $L>0$ be fixed and let $\zeta = (\zeta_0,\dots,\zeta_{N+2})\in \real^{N+3}$ be a given vector such that
\begin{equation}\label{E:zetacondition}
\vz_0>\dots>\vz_{N+1}>\vz_{N+2}=0. 
\end{equation}
We define
$\vx{i} \in \real^{N+1}$, $\vg{i} \in \real^{N+1}$, $\vf{i} \in \real$ for $i\in \{0,\dots,N+1\}$ as follows:
\begin{equation}\label{E:wcfinterpolation}
\begin{aligned}
	& \vx{i} := -\sum_{j=0}^{i-1} \frac{\zeta_j -\zeta_{i+1}}{\sqrt{\zeta_j -\zeta_{j+1}}} \ee{j}, \quad i=0,\dots,N+1, \\
	& \vg{i} := L\sqrt{\zeta_i -\zeta_{i+1}} \ee{i}, \quad i=0,\dots,N,\\
	& \vg{N+1} := 0, \\
	& \vf{i} := \frac{L}{2}(\zeta_i + \zeta_{i+1}), \quad i=0,\dots,N,\\
	& \vf{N+1} := 0.
\end{aligned}
\end{equation}
In addition, 
taking
$\mathcal{T}^\zeta := \{(\vx{i}, \vg{i}, \vf{i})\}_{i\in \{0,\dots,N+1\}}$,  we 
set
\[
 \vw(y, \nu,\alpha) := w_{\mathcal{T}^\zeta}(y, \nu,\alpha),
\]
and 
\[
	\vW(y):=W_{\mathcal{T}^\zeta}^{\real^{N+1}_+}(y)  \left( \equiv \min_{\nu \in \real^{N+1}_+, \alpha\in \Delta_{N+2}}  \vw(y, \nu,\alpha)\right).
\]
\end{definition}

We intentionally leave the values of $\zeta$ undefined at this point, as most properties of the function $\vW$ can be established without introducing additional requirements on $\zeta$.
The specific values used in the lower complexity proof will be defined towards the end of this section.

The rest of this section is dedicated to the investigation of the special properties of $\vW(y)$.
We start with a few technical properties.
\begin{proposition}
The  following identities hold:
\begin{align}
	& f^\zeta_i -\frac{1}{2L} \| g^\zeta_{i}\|^2 = L \zeta_{i+1}, \qquad i=0,\dots,N+1, \label{E:fnormaval}\\
	& - \vx{i} + \frac{1}{L}\vg{i} =\sum_{j=0}^{i} \frac{\zeta_j -\zeta_{i+1}}{\zeta_j -\zeta_{j+1}} \vg{j} \in \real^{N+1}_+, \quad i=0,\dots,N+1, \label{E:xiequiv} \\
	&  \langle \vg{i}, -\vx{j}+ \frac{1}{L}\vg{j}\rangle = L \max(\zeta_i-\zeta_{j+1}, 0), \quad i,j=0,\dots,N+1. \label{E:maxdot}
\end{align}
\end{proposition}
\begin{proof}
\eqref{E:fnormaval}--\eqref{E:maxdot} follow directly from the definitions of $\vx{i}$, $\vg{i}$ and $\vf{i}$.
\end{proof}

The following corollary establishes the convexity, smoothness and interpolating properties of $\vW$.
In particular, it implies that $\nabla \vW(\vx{N+1})=\vg{N+1}=0$, hence
 $\vx{N+1}$ is a minimizer for $\vW$, and thus ${\vW}^* = \vf{N+1}=0$.
\begin{corollary}\label{C:vwintarpolating}
The function $\vW$ is convex and has a Lipschitz-continuous gradient with constant $L$.
In addition, $\vW$ is interpolating through $\mathcal{T}^\zeta$, i.e.,
$\vW(\vx{i})=\vf{i}$ and $\nabla \vW(\vx{i})=\vg{i}$ for all $i\in \{0,\dots,N+1\}$.
\end{corollary}
\begin{proof}
By Theorem~\ref{T:basicprop}, it is enough to verify that $\vx{i}$, $\vf{i}$, and $\vg{i}$ satisfy conditions \eqref{E:conic} and \eqref{E:Lipchitzxi} for all $i\in  \{0,\dots,N+1\}$.

Condition \eqref{E:conic} follows immediately from the properties of the projection operator on the nonnegative orthant.
In order to establish~\eqref{E:Lipchitzxi}, 
from the definitions \eqref{E:wcfinterpolation}
it remains to show that for any $i\neq j \in  \{0,\dots,N+1\}$, 
\begin{align*}
& \frac{L}{2} ((\zeta_i-\zeta_{i+1}) + (\zeta_j-\zeta_{j+1})) \leq \frac{L}{2}(\zeta_j+\zeta_{j+1})-\frac{L}{2}(\zeta_i+\zeta_{i+1}) - \langle \vg{i}, \vx{j}-\vx{i}\rangle
\end{align*}
which, by $\langle \vg{i}, \vx{i}\rangle=\langle \vg{i}, \frac{1}{L}\vg{j}\rangle=0$ for $i\neq j$ reduces to
\begin{align*}
& L(\zeta_i -\zeta_{j+1}) \leq  \langle \vg{i}, -\vx{j}+\frac{1}{L}\vg{j}\rangle 
\end{align*}
and follows from~\eqref{E:maxdot}.
\end{proof}

The next lemma identifies an important family of optimal solutions to the convex optimization problem $\vW$.
\begin{lemma}\label{L:alphaform}
For any $y\in \real^{N+1}$, there exists an optimal solution $(\nu^*, \alpha^*)$ to the optimization problem $\vW(y)$ 
where $\alpha^*\in\Delta_{N+2}$ is of the form
\begin{equation}\label{E:alphaform}
	\alpha^* = (\underbrace{0,\dots,0}_{m\text{ times}} ,\alpha^*_m,1-\alpha^*_m,\underbrace{0,\dots,0}_{N-m\text{ times}})
\end{equation}
for some $m\in\{0,\dots,N\}$.
\end{lemma}
\begin{proof}
We proceed by showing that 
for any $(\nu, \alpha)\in \real^{N+1}_+ \times \Delta_{N+2}$ that is feasible for the optimization problem $\vW(y)$,
if $\alpha_{i_1}>0$ and $\alpha_{i_3}>0$
for some $0\leq i_1 < i_2 < i_3\leq N+1$
then it is possible 
to decrease either $\alpha_{i_1}$ or $\alpha_{i_3}$ to zero while increasing $\alpha_{i_2}$ 
without affecting the objective or violating the constraints.

Let $\delta\in\real^{N+2}$ be defined by
\begin{align*}
	\delta_{i_1} &= (\vf{i_2} -\frac{1}{2L} \| \vf{i_2}\|^2)-(\vf{i_3} -\frac{1}{2L} \| \vf{i_3}\|^2) = L(\zeta_{i_2+1} - \zeta_{i_3+1}),\\
	\delta_{i_2} &= (\vf{i_3} -\frac{1}{2L} \| \vf{i_3}\|^2)-(\vf{i_1} -\frac{1}{2L} \| \vf{i_1}\|^2) = L(\zeta_{i_3+1} - \zeta_{i_1+1}),\\
	\delta_{i_3} &= (\vf{i_1} -\frac{1}{2L} \| \vf{i_1}\|^2)-(\vf{i_2} -\frac{1}{2L} \| \vf{i_2}\|^2) = L(\zeta_{i_1+1} - \zeta_{i_2+1}),
\end{align*}
with $\delta_i=0$ for $i \notin \{i_1,i_2,i_3\}$, and, in addition, set
\begin{align*}
	& t := \min \left(\frac{\alpha_{i_1}}{\delta_{i_1}}, \frac{\alpha_{i_3}}{\delta_{i_3}}\right), \\
	&\nu' := \nu-t \sum_{i=0}^{N+1} \delta_i  (\vx{i}-\frac{1}{L}\vg{i}), \\
	&\alpha' := \alpha -t \delta,
\end{align*}
then clearly either  $\alpha'_{i_1}=0$ or $\alpha'_{i_3}=0$ and it remains to show that $(\nu',\alpha')$ is feasible for $\vW$ and attains same objective value, i.e., (i) $\nu'\in \real^{N+1}_+$, (ii) $\alpha' \in \Delta_{N+2}$, and (iii) $\vw(y, \nu',\alpha') = \vw(y, \nu,\alpha)$.

(i) In order to establish that $\nu'\in \real^{N+1}_+$, since $\nu\in \real^{N+1}_+$ and $t>0$, it is enough to show that
\[
	-\sum_{i=0}^{N+1} \delta_i (\vx{i}-\frac{1}{L}\vg{i}) \in \real^{N+1}_+.
\]
From \eqref{E:xiequiv} and the definition of $\delta$, we have
\begin{align*}
 &  -\sum_{i=0}^{N+1} \delta_i (\vx{i}-\frac{1}{L}\vg{i}) = \sum_{k=1,2,3} \delta_{i_k} \sum_{j=0}^{i_k} \frac{\zeta_j -\zeta_{i_k+1}}{\zeta_j -\zeta_{j+1}} \vg{j}\\
 & = \sum_{j=0}^{i_1} \frac{\delta_{i_1}(\zeta_j -\zeta_{i_1+1})+\delta_{i_2}(\zeta_j -\zeta_{i_2+1})+\delta_{i_3}(\zeta_j -\zeta_{i_3+1})}{\zeta_j -\zeta_{j+1}} \vg{j} \\
 & \quad + \sum_{j=i_1+1}^{i_2} \frac{\delta_{i_2}(\zeta_j -\zeta_{i_2+1})+\delta_{i_3}(\zeta_j -\zeta_{i_3+1})}{\zeta_j -\zeta_{j+1}} \vg{j}
+\sum_{j=i_2+1}^{i_3} \frac{\delta_{i_3}(\zeta_j -\zeta_{i_3+1})}{\zeta_j -\zeta_{j+1}} \vg{j},
\end{align*}
then since the vectors $g_j$ are in $\real^{N+1}_+$, it is sufficient to verify that for all $j$ the coefficient of $g_j$ is nonnegative.
We consider three cases: 
first, suppose $j\leq i_1$, then the coefficient of $\vg{j}$ is:
\begin{align*}
& L(\zeta_{i_2+1} -\zeta_{i_3+1} )\frac{\zeta_j -\zeta_{i_1+1}}{\zeta_j -\zeta_{j+1}}
+ L(\zeta_{i_3+1} -\zeta_{i_1+1} )\frac{\zeta_j -\zeta_{i_2+1}} {\zeta_j -\zeta_{j+1}}
+ L(\zeta_{i_1+1} -\zeta_{i_2+1} )\frac{\zeta_j -\zeta_{i_3+1}} {\zeta_j -\zeta_{j+1}} = 0.
\end{align*}
Next, suppose $i_1<j\leq i_2$, then  the coefficient of $\vg{j}$ is
\begin{align*}
& L(\zeta_{i_3+1} -\zeta_{i_1+1} )\frac{\zeta_j -\zeta_{i_2+1}} {\zeta_j -\zeta_{j+1}}
+ L(\zeta_{i_1+1} -\zeta_{i_2+1} )\frac{\zeta_j -\zeta_{i_3+1}} {\zeta_j -\zeta_{j+1}} = -L(\zeta_{i_2+1} -\zeta_{i_3+1} )\frac{\zeta_j -\zeta_{i_1+1}}{\zeta_j -\zeta_{j+1}},
\end{align*}
which is positive since $\{\zeta_i\}$ is strictly decreasing.
Lastly, for the case $i_2<j\leq i_3$, the coefficient of $\vg{j}$ is
\begin{align*}
& L(\zeta_{i_1+1} - \zeta_{i_2+1})\frac{\zeta_j -\zeta_{i_3+1}}{\zeta_j -\zeta_{j+1}},
\end{align*}
which is again positive from the monotonicity of $\zeta$.

(ii) $\alpha' \in \Delta_{N+2}$. 
Firstly, the nonnegativity of $\alpha'$ follows from the choice of $t$ and since $\delta_{i_1}, \delta_{i_3}>0$ while $\delta_{i_2}<0$. 
Secondly, from the definition of $\alpha'$ and
\[
	\sum_{i=0}^{N+1} \delta_i = \delta_{i_1} + \delta_{i_2} + \delta_{i_3} = 0,
\]
and we get that $\sum_{i=0}^{N+1} \alpha'_i=1$,

(iii) $\vw(y, \nu',\alpha') = \vw(y, \nu,\alpha)$. By the definition of $\vw$ it is enough to show that
\begin{align*}
	& y+\nu-\sum_{i=0}^{N+1} \alpha_i  (\vx{i}-\frac{1}{L}\vg{i}) = y+\nu'-\sum_{i=0}^{N+1} \alpha'_i  (\vx{i}-\frac{1}{L}\vg{i}) \\
	\text{and}\quad & \sum_{i=0}^{N+1} \alpha_i (\vf{i} -\frac{1}{2L} \| \vg{i}\|^2) = \sum_{i=0}^{N+1} \alpha'_i (\vf{i} -\frac{1}{2L} \| \vg{i}\|^2).
\end{align*}
Indeed, the first equality follows from 
\[
	\nu-\nu' = t \sum_{i=0}^{N+1} \delta_i  (\vx{i}-\frac{1}{L}\vg{i}) = \sum_{i=0}^{N+1} (\alpha_i -\alpha'_i )  (\vx{i}-\frac{1}{L}\vg{i}),
\]
and the second equality follows from
\begin{align*}
	\sum_{i=0}^{N+1}& (\alpha_i -\alpha'_i ) (\vf{i} -\frac{1}{2L} \| \vg{i}\|^2) =\sum_{i=0}^{N+1} t\delta_i (\vf{i} -\frac{1}{2L} \| \vg{i}\|^2) \\
	 = &  t \left((\vf{i_2} -\frac{1}{2L} \| \vf{i_2}\|^2)-(f_{i_3} -\frac{1}{2L} \| \vf{i_3}\|^2)\right)(\vf{i_1} -\frac{1}{2L} \| \vf{i_1}\|^2)\\
	& + t \left( (\vf{i_3} -\frac{1}{2L} \| \vf{i_3}\|^2)-(f_{i_1} -\frac{1}{2L} \| \vf{i_1}\|^2) \right) (\vf{i_2} -\frac{1}{2L} \| \vf{i_2}\|^2) \\
	& + t \left( (\vf{i_1} -\frac{1}{2L} \| \vf{i_1}\|^2)-(f_{i_2} -\frac{1}{2L} \| \vf{i_2}\|^2) \right) (\vf{i_3} -\frac{1}{2L} \| \vf{i_3}\|^2)= 0.
\end{align*}

\medskip

Repeatedly applying the procedure described above on an optimal solution of $\vW$, $(\nu^*, \alpha^*)$,
we get that the distance between the first and last positive elements of $\alpha_*$ can always be decreased as long as it is greater than one, which completes the proof.
\end{proof}

Note that the previous result allows for an efficient numerical solution of $\vW(y)$ via $N+1$ independent one-dimensional convex 
problems 
instead of a single $N+1$-dimensional problem.
The following theorem shows that in some cases it is possible to further reduce the number of subproblems that need to be solved in order to find a solution to $\vW(y)$.
More importantly, the following property turns out to be fundamental and will be later used to establish the main properties of $\vW$.
\begin{theorem}\label{T:oneIteration}
Let $\bar y=(\bar y_0,\dots,\bar y_N)$ be a vector in $\real^{N+1}$.
If $\bar y_k\geq 0$ for some $k\in\{0,\dots, N\}$, then $\vW(\bar y)$ has an optimal solution $(\nu^*, \alpha^*)$, where $\alpha^*$ is of the form~\eqref{E:alphaform} and, in addition, satisfies $\alpha^*_{k+1}=\dots=\alpha^*_{N+1}=0$.
\end{theorem}

\begin{proof}
Let $(\nu^*, \alpha^*)$ be an optimal solution to $\vW(\bar y)$ where
$\alpha^*$ is of the form~\eqref{E:alphaform} for some $0\leq m\leq N$.
Further assume that $\alpha^*$ is chosen so that the value of $m$ is minimal among all such optimal solutions to $\vW(\bar y)$.

In order to simplify the notations in the following,
we denote 
\[
 \tmpvec:=\alpha^*_m (-\vx{m}+\frac{1}{L}\vg{m})+(1-\alpha^*_m) (-\vx{m+1}+\frac{1}{L}\vg{m+1}).
\]
Now, suppose that $\alpha^*$ is of the form  \eqref{E:alphaform} with $m \geq k$ and $\alpha^*_{m}<1$,
then from $\alpha^*_{m}<1$, \eqref{E:maxdot}, and \eqref{E:zetacondition} we have
\begin{equation}\label{E:strictineq}
	\langle \vg{k}, \tmpvec \rangle > L(\zeta_k -\zeta_{k+1}).
\end{equation}
In addition, the closed-form solution for $\nu^*$ \eqref{E:nuoptcond} implies that
\begin{equation}\label{E:lemmaoneIterationcondition}
	\nabla_\nu  \vw(\bar y,\nu^*, \alpha^*) = \bar y+ \nu^*+\tmpvec  \in \real^{N+1}_+,
\end{equation}
and, from the condition on $\bar y_k$, we have
\begin{equation}\label{E:lemmaoneIterationconditionb}
	\langle \vg{k}, \bar y+ \nu^*\rangle = L \sqrt{\zeta_k-\zeta_{k+1}}(\bar y_k +\nu^*_k) \geq 0.
\end{equation}

Combining these results we get
\begin{align*}
& \frac{\partial \vw(\bar y,\nu^*, \alpha^*)}{\partial \alpha_{k}}  = \langle -\vx{k}+\frac{1}{L}\vg{k}, \bar y+ \nu^*+\tmpvec \rangle  +\vf{k} -\frac{1}{2L} \| \vg{k}\|^2 \\
& = \langle \sum_{j=0}^{k-1} \frac{\zeta_j -\zeta_{k+1}}{\zeta_j -\zeta_{j+1}} \vg{j}, \bar y+ \nu^*+\tmpvec \rangle +\langle \vg{k}, \bar y+ \nu^* + \tmpvec\rangle +L \zeta_{k+1} \\
& = \langle \sum_{j=0}^{k-1} \frac{\zeta_j -\zeta_{k+1}}{\zeta_j -\zeta_{j+1}} \vg{j}, \bar y+ \nu^*+\tmpvec \rangle +\langle \vg{k}, \bar y+ \nu^* + \tmpvec\rangle + L\frac{\zeta_{k+1} -\zeta_{m+2}}{\zeta_k -\zeta_{k+1}} (\zeta_k -\zeta_{k+1}) + L \zeta_{m+2} \\
& < \langle \sum_{j=0}^{k-1} \frac{\zeta_j -\zeta_{k+1}}{\zeta_j -\zeta_{j+1}} \vg{j}, \bar y+ \nu^*+\tmpvec \rangle +\langle \vg{k}, \bar y+ \nu^* + \tmpvec\rangle + \frac{\zeta_{k+1} -\zeta_{m+2}}{\zeta_k -\zeta_{k+1}} \langle \vg{k}, \tmpvec\rangle +L \zeta_{m+2} \\
& \leq \langle \sum_{j=0}^{k-1} \frac{\zeta_j -\zeta_{k+1}}{\zeta_j -\zeta_{j+1}} \vg{j}, \bar y + \nu^*+\tmpvec\rangle +\frac{\zeta_k -\zeta_{m+2}}{\zeta_k -\zeta_{k+1}}\langle \vg{k}, \bar y + \nu^*\rangle +\frac{\zeta_k -\zeta_{m+2}}{\zeta_k -\zeta_{k+1}}\langle \vg{k}, \tmpvec\rangle +L \zeta_{m+2} \\
& \leq \langle \sum_{j=0}^{k} \frac{\zeta_j -\zeta_{m+2}}{\zeta_j -\zeta_{j+1}} \vg{j}, \bar y+ \nu^*+\tmpvec \rangle +L \zeta_{m+2} \\
& \leq \langle \sum_{j=0}^{m+1} \frac{\zeta_j -\zeta_{m+2}}{\zeta_j -\zeta_{j+1}} \vg{j}, \bar y+ \nu^*+\tmpvec \rangle +\vf{m+1} -\frac{1}{2L} \| \vg{m+1}\|^2 \\
& = \frac{\partial \vw(\bar y,\nu^*, \alpha^*)}{\partial \alpha_{m+1}},
\end{align*}
where the strict inequality follows from~\eqref{E:strictineq} and the following inequalities follow from \eqref{E:lemmaoneIterationconditionb}, 
\eqref{E:zetacondition},
and \eqref{E:lemmaoneIterationcondition}, respectively.
By the first-order optimality conditions \eqref{E:forderoptalpha} it then follows that $\alpha^*_{m+1}=0$, i.e., $\alpha_{m}=1$, reaching a contradiction.
We therefore conclude that either $m < k$ or $\alpha^*_{m}=1$.

We are left with the following three cases:
\begin{itemize}
\item $m < k$: Here $\alpha^*$ is in the claimed form.
\item $m =  k = 0$ and $\alpha^*_{0}=1$: Again, $\alpha^*$ is in the required form.
\item $m>0$, $m \geq  k$ , and $\alpha^*_{m}=1$: This case is not consistent with the minimality of $m$ since it is possible to 
represent $\alpha^*$ using a smaller value of $m$ by
taking $m'=m-1$ with $\alpha^*_{m-1}=0$.
\end{itemize}
We conclude that for all valid cases $\alpha^*$ is in the desired form.
\end{proof}

The following two corollaries will be the main building blocks of the proof in the next section. 
\begin{corollary}\label{C:minvalue}
Suppose $\bar y=(\bar y_0,\dots,\bar y_N)$ is a vector in $\real^{N+1}$ such that $\bar y_k\geq 0$ for some $k\in \{0,\dots,N\}$,
then
\[
	\vW(\bar y) - {\vW}^* \geq \vf{k} 
\]
\end{corollary}
\begin{proof}
Denote\phantomsection\label{P:realkn}
\[
	\realkn := \{ (y_0,\dots,y_N)\in \real^{N+1} : y_k=\dots= y_N=0\},
\]
and consider the case $\bar y \in \real^{k,N+1}$.
Since for all $v\in \real^{k,N+1}$ we have
$\langle \nabla \vW(\vx{k}), v\rangle = \langle \vg{k}, v\rangle = L\sqrt{\zeta_i -\zeta_{i+1}} \langle  e_{k+1}, v\rangle=0$, 
then clearly $\vx{k}$ minimizes $\vW(\bar y)$ over $\real^{k,N+1}$,
and we get that $\vW(\bar y)- {\vW}^*=\vW(\bar y) \geq \vW(\vx{k})= \vf{k}$.

For the general case $\bar y \in \real^{N+1}$, let $(\nu^*, \alpha^*)$ be an optimal solution to $\vW(\bar y)$
such that $\alpha^*_{k+1}=\dots=\alpha^*_{N+1}=0$ (by Theorem~\ref{T:oneIteration} such a solution exsits),
and denote $\hat y=(\bar y_0,\dots,\bar y_{k-1}, 0,\dots,0)$,
$\hat \nu^* = (\nu^*_0,\dots,\nu^*_{k-1}, 0,\dots,0)$. We get
\begin{align*}
	& \vW(\bar y)- {\vW}^* =  \vW(\bar y) = \vw(\bar y,\nu^*, \alpha^*)  \\
	& = \frac{L}{2} \|\bar y+\nu^*-\sum_{i=0}^k \alpha^*_i  (\vx{i}-\frac{1}{L}\vg{i}) \|^2 +\sum_{i=0}^k \alpha^*_i (\vf{i} -\frac{1}{2L} \| \vg{i}\|^2) \\
	& = \frac{L}{2} \|\hat y+\hat \nu^*+ \sum_{i=k}^N (\bar y_i+\nu^*_i) \ee{i}-\sum_{i=0}^k \alpha^*_i  (\vx{i}-\frac{1}{L}\vg{i}) \|^2  +\sum_{i=0}^k \alpha^*_i (\vf{i} -\frac{1}{2L} \| \vg{i}\|^2)  \\
	& = \frac{L}{2} \|\hat y+\hat \nu^*+(\bar y_k+\nu^*_k) \ee{k}-\sum_{i=0}^k \alpha^*_i  (\vx{i}-\frac{1}{L}\vg{i}) \|^2  + \sum_{i=k+1}^N (\bar y_i + \nu^*_i)^2 +\sum_{i=0}^k \alpha^*_i (\vf{i} -\frac{1}{2L} \| \vg{i}\|^2)  \\
	& \geq \frac{L}{2} \|\hat y+\hat \nu^*-\sum_{i=0}^k \alpha^*_i  (\vx{i}-\frac{1}{L}\vg{i}) \|^2  + \sum_{i=k+1}^N (\bar y_i + \nu^*_i)^2 +\sum_{i=0}^k \alpha^*_i (\vf{i} -\frac{1}{2L} \| \vg{i}\|^2)  \\
	& \geq  \vw(\hat y,\hat \nu^*, \alpha^*) \geq \vW(\hat y) \geq \vf{k},
\end{align*}
where the first inequality follows from $\bar y_k+\nu^*_k \geq 0$ and \eqref{E:xiequiv}.
\end{proof}

\begin{corollary}\label{C:zerodrivative}
Suppose $\bar y=(\bar y_0,\dots,\bar y_N)$ is a vector such that $\bar y_k\geq 0$ and $\bar y_n=0$ for some $0\leq k < n\leq N$,
then
\[
	\frac{\partial}{\partial y_n} \vW(\bar y) = 0.
\]
\end{corollary}
\begin{proof}
Let $(\nu^*, \alpha^*)$ be an optimal solution of $\vW(\bar y)$ in the form guaranteed by Theorem~\ref{T:oneIteration},
then
\[
	\nabla \vW (\bar y) = L  \left(\bar y+\nu^*-\sum_{i=0}^{k} \alpha^*_i  (\vx{i}-\frac{1}{L}\vg{i})\right),
\]
and
we get
\begin{align*}
	& \frac{\partial}{\partial y_n} \vW(\bar y)
	= L \langle \bar y+\nu^*-\sum_{i=0}^{k} \alpha^*_i  (\vx{i}-\frac{1}{L}\vg{i}), \ee{n} \rangle = 0,
\end{align*}
where the last equality follows from $\bar y_n=0$,  
\eqref{E:nuoptcond},
and since $\langle \vx{i}-\frac{1}{L}\vg{i}, \ee{n}\rangle=0$ for any $i<n$, 
\end{proof}

\begin{remark}
As an immediate result of Corollary~\ref{C:zerodrivative}, it follows that
\[
  y\in \realkn \Rightarrow \nabla \vW(y)\in\real^{k+1,N+1}, \quad k=0,\dots,N.
\]
Using a simple inductive argument, we get that for any sequence $x_0,\dots,x_N\in \real^{N+1}$ 
starting at $x_0=0$ and satisfying (c.f.,~\cite[Assumption~2.1.4]{nest-book-04})
\begin{equation}\label{E:nestassumption}
	x_k \in \linspan\{ \nabla \vW(x_0),\dots, \nabla \vW(x_{k-1})\}, \quad k=1,\dots,N,
\end{equation}
we have $x_k\in \realkn$, $k=0,\dots,N$,
and therefore, by Corollary~\ref{C:minvalue}, we get that the inequality
\begin{equation}\label{E:notelowerbound}
	\vW(x_k)-{\vW}^* \geq \vf{k}, \quad k=0,\dots,N
\end{equation}
holds for any such sequence.
We conclude that
the worst-case absolute inaccuracy 
of any first-order method which generates sequences that satisfy~\eqref{E:nestassumption}
cannot be lower than the RHS of~\eqref{E:notelowerbound},
i.e., this inequality forms a lower-bound on the efficiency estimate of such methods.
As noted by Nesterov~\cite[Page~59]{nest-book-04}, this 
limitation on the structure of the first-order methods for which this bound is applicable can avoided by some additional reasoning, 
however, this requires introducing some additional assumptions, e.g.,
that the function $\vW$ can be embedded in a vector space whose dimension, $d$, is at least $d\geq 2N+1$.
By taking advantage of the properties derived above, we show in the next section that the
lower bound \eqref{E:notelowerbound} applies for sequences generated by any first-order method for the case $d\geq N+1$.
\end{remark}

\medskip

At this stage, the exact value of $\zeta$ has not yet been defined.
In order to find an instance of $\zeta$ which yields the best 
lower complexity bounds on first-order methods,
one needs to 
maximize $\vf{N}$ while keeping $\|\vx{N+1}\|$ bounded.
Solving this problem, we obtain the following solution instance
(we skip the optimality proof, as it follows from the tightness of the resulting bound):
\begin{definition}
Suppose $R>0$ and $N\in \mathbb{N}$ are given. 
We define $\sz=(\sz_0,\dots,\sz_{N+2})$ by
\begin{equation}\label{E:szdef}
\begin{aligned}
	& \sz_i=\frac{2\theta_i}{2\theta_i-1}\sz_{i+1}, \quad i=0,\dots,N-1,\\
	& \sz_N = \frac{\theta_N}{\theta_N-1}\sz_{N+1}, \\
	& \sz_{N+1} = \frac{\theta_N-1}{\theta_N^2(2\theta_N-1)} R^2, \\
	& \sz_{N+2} = 0,
\end{aligned}
\end{equation}
where $\theta_i$ is defined as in \eqref{E:thetadef}.
\end{definition}

\begin{lemma}\label{L:xstardist}
Consider \eqref{E:wcfinterpolation}, taking $\zeta=\sz$, then
\begin{align*}
	& \|\sx{N+1}\|=R \quad \text{and} \quad \sff{N} = \frac{LR^2}{2\theta_N^2}.
\end{align*}
\end{lemma}
\begin{proof}

Using the definition of $\sz_j$ and the relation
\begin{align*}
	& \theta_{i-1}^2 = \theta_i (\theta_i-1), \quad i=1,\dots,N-1,
\end{align*}
for $j=0,\dots, N-1$ we have
\begin{align*}
  & \sz_j \theta_j = \sz_N \theta_j \prod_{i=j}^{N-1} \frac{2\theta_i}{2\theta_i-1}  \\
  & = \sz_N\frac{2\theta_j^2}{2\theta_j-1} \prod_{i=j+1}^{N-1} \frac{2\theta_i}{2\theta_i-1} =\sz_N \frac{2 \theta_{j+1} -2}{2\theta_j-1} \theta_{j+1} \prod_{i=j+1}^{N-1} \frac{2\theta_i}{2\theta_i-1}\\
  & = \sz_N \frac{2 \theta_{j+1} - 2}{2\theta_j-1} \frac{2\theta_{j+1}^2}{2\theta_{j+1}-1} \prod_{i=j+2}^{N-1} \frac{2\theta_i}{2\theta_i-1} = \sz_N \frac{2 \theta_{j+1} - 2}{2\theta_j-1} \frac{2 \theta_{j+2} -2}{2\theta_{j+1}-1} \theta_{j+2} \prod_{i=j+2}^{N-1} \frac{2\theta_i}{2\theta_i-1} \\
  & = \dots = \sz_N \prod_{i=j}^{N-2} \frac{2 \theta_{i+1} - 2}{2\theta_i-1} \cdot \theta_{N-1} \frac{2\theta_{N-1}}{2\theta_{N-1}-1},
\end{align*}
then, using $\theta_{N-1}^2 = \frac{1}{2}\theta_{N} (\theta_{N}-1)$, we reach
\begin{align}
 \sz_j \theta_j = \frac{\sz_N \theta_N}{2} \prod_{i=j}^{N-1} \frac{2\theta_{i+1}-2}{2\theta_i-1}, \quad j=0,\dots, N-1. \label{E:szthetaid}
\end{align}
Next, from the definition of $\sx{N+1}$, \eqref{E:szdef}, and \eqref{E:szthetaid} we get
\begin{align*}
	& \|\sx{N+1}\|^2 =\sum_{j=0}^{N} \frac{(\sz_j)^2 }{\sz_j -\sz_{j+1}}=\sum_{j=0}^{N-1} \frac{\sz_j }{1 -\frac{2\theta_i-1}{2\theta_i}} +\frac{\sz_N}{1 -\frac{\theta_N-1}{\theta_N}} = \sum_{j=0}^{N-1} 2 \sz_j\theta_j +  \sz_N\theta_N \\
	& = \sum_{j=0}^{N-1} \sz_N  \theta_N\prod_{i=j}^{N-1} \frac{2\theta_{i+1}-2}{2\theta_i-1} +\sz_N \theta_N \\
	& = \theta_N \sz_N \left(\frac{2\theta_{N}-2}{2\theta_{N-1}-1}\left(\frac{2\theta_{N-1}-2}{2\theta_{N-2}-1}\left(\dots \left( \frac{2\theta_{2}-2}{2\theta_{1}-1} \left(\frac{2\theta_{1}-2}{2\theta_{0}-1}+1\right)+1\right)\dots\right)+1 \right) +1\right),
\end{align*}
which from $2\theta_0-1=1$ turns out to be a telescopic product and reduces to
\begin{align*}
	&  \|\sx{N+1}\|^2 =\theta_N \sz_N (2\theta_{N}-2)= \frac{\theta_N^2(2\theta_N-1)}{\theta_N-1} \sz_{N+1} =R^2.
\end{align*}

For the second part of the claim, we have
\begin{align*}
    & \sff{N} = \frac{L}{2}(\sz_N+\sz_{N+1}) \\
	& = 	\frac{L}{2} \left( \frac{\theta_N}{\theta_N-1}\frac{\theta_N-1}{\theta_N^2(2\theta_N-1)} R^2+ \frac{\theta_N-1}{\theta_N^2(2\theta_N-1)} R^2\right)\\
	& = 	\frac{L R^2}{2} \left( \frac{\theta_N}{\theta_N^2(2\theta_N-1)} + \frac{\theta_N-1}{\theta_N^2(2\theta_N-1)}\right)\\
	& = \frac{L R^2}{2\theta_N^2},
\end{align*}
which establishes the claim.
\end{proof}

\section{The main result}\label{S:lower}
We now build upon the results of the previous sections to derive the exact minimax risk associated with smooth and convex minimization.

We start by introducing a family of functions that
will be used as a basis for proving the main result.
\begin{definition}
Let $N, d\in \mathbb{N}$ be such that $N\leq d$.
For an orthonormal set $\{v_0,\dots,v_N\}\subset \real^d$, let
$W_{v_0,\dots,v_N}:\real^d\rightarrow\real$ be defined by
\[
    W_{v_0,\dots,v_N}(z) := \sW(\langle z, v_0\rangle,\dots,\langle z, v_N\rangle).
\]
\end{definition}
\begin{lemma}\label{L:funcinvariance}
Let $N\leq d$ and
suppose $\{v_0,\dots,v_N\}$ and $\{u_0,\dots,u_N\}$ are two
orthonormal sets of vectors in $\real^d$ with $v_0=u_0,\dots,v_k=u_k$ for some $k\leq N$, 
then for any $z\in \real^d$ such that $z\in \linspan\{v_0,\dots,v_k\}$ and $\langle z, v_k\rangle\geq 0$ we have
\[
 W_{v_0,\dots,v_N}(z)=W_{u_0,\dots,u_N}(z) \quad \text{ and } \quad \nabla W_{v_0,\dots,v_N}(z)=\nabla W_{u_0,\dots,u_N}(z).
\]
\end{lemma}
\begin{proof}
Since $\langle z, u_{k+1}\rangle=\langle z, v_{k+1}\rangle=\dots=\langle z, u_N\rangle=\langle z, v_N\rangle=0$, we get
\begin{align*}
     W_{v_0,\dots,v_N}(z) & = \sW(\langle z, v_0\rangle,\dots,\langle z, v_k\rangle,0,\dots, 0)\\
     &= \sW(\langle z, u_0\rangle,\dots,\langle z, u_k\rangle,0,\dots, 0) = W_{u_0,\dots,u_N}(z),
\end{align*}
which establishes the first part of the claim.
For the second part of the claim, we have
\[
	\frac{\partial}{\partial v_n} W_{v_0,\dots,v_N}(z) = \frac{\partial}{\partial u_n} W_{u_0,\dots,u_N}(z) ,
\]
which is immediate for $n\leq k$ and follows 
from Corollary~\ref{C:zerodrivative} for $n>k$.
\end{proof}

We are now ready to state the main result of this paper: a lower bound on the minimax risk associated with $\finstances$.
\begin{theorem}\label{T:lowerbound}
Let $L, R>0$ and $N\in \mathbb{N}$,
then for any first-order method $\algorithm$ that performs at most $N$ calls to its first-order oracle
and any $d\geq N+1$
there exists 
a convex function
$w_\algorithm \in C^{1,1}_L(\real^d)$ 
and
$x_0\in \real^d$ 
such that
\begin{equation}\label{E:lowerboundfn}
\begin{aligned}
    & \|x_*-x_0\|\leq R, \quad \text{ for some $x_*\in X_*(w_A)$}, \\
	\text{and} \quad & w_\algorithm(\algorithm(\oraclea, x_0))-w_\algorithm^*\geq  \frac{LR^2}{2\theta_N^2},
\end{aligned}
\end{equation}
where $\theta_N$ is as defined in \eqref{E:thetadef}.
\end{theorem}
\begin{proof}
Since a first-order method can only gain information on the objective through 
the first-order information at the selected search points,
it cannot distinguish between functions that have identical
first-order information at these points,
and will execute precisely in the same way on such functions.
We can therefore maintain during the run of the method
a set of functions that have identical first-order information at
the points selected so far,
postponing the choice of a specific objective until the method has completed its run.




We set $x_0=0$ as the reference point given to algorithm $A$.
Now, suppose $\ax_k\in\real^d$ is the $(k+1)$th search point chosen by algorithm $\algorithm$, $0\leq k\leq N-1$. At this stage, we
take a vector $\bar v_k\in \real^d$ that satisfies the following conditions:
\begin{itemize}
\item $\bar v_k$ is a unit vector orthogonal to $\{\bar v_0,\dots,\bar v_{k-1}\}$,
\item $\ax_k \in \linspan\{\bar v_0,\dots,\bar v_k\}$,
\item $\langle \ax_k, \bar v_k\rangle \geq 0$,
\end{itemize}
and set 
\[
	\Omega_k:=\{ W_{v_0,\dots,v_N} : \text{$v_0,\dots,v_N$ are orthonormal, } v_0 = \bar v_0,\dots,v_k=\bar v_k \}.
\]
By Lemma~\ref{L:funcinvariance}, all functions in $\Omega_k$ share the same
first-order information at $\ax_k$ (and also at $\ax_0,\dots,\ax_{k-1}$, since $\Omega_k\subset \Omega_{k-1}\subset \dots \subset \Omega_0$),
consequently, algorithm $A$ cannot differentiate between these functions, and 
we can postpone the specific choice of function for a later stage.

After invoking the oracle $N$ times,
suppose algorithm $\algorithm$ has selected $\ax_N\in\real^d$ as its approximate solution.
Let $\bar v_N$ be a unit vector orthogonal to $\{\bar v_0,\dots,\bar v_{N-1}\}$
such that $\langle \ax_N, \bar v_N\rangle\geq 0$,
and set 
\[
	w_\algorithm:= W_{\bar v_0,\dots,\bar v_N},
\]
then by the construction above,
algorithm $\algorithm$, when applied on $w_\algorithm$ with the reference point $x_0$,
will generate the 
search points $\ax_0,\dots,\ax_{N-1}$
and it return as its output the vector $\ax_N$.

To complete the proof, note that
 $w_\algorithm$ is of the form $\sW(Qy)$ for some matrix $Q\in\real^{N+1 \times d}$ with orthonormal rows,
hence $w_\algorithm$ shares the convexity, smoothness, and Lipschitz constant of $\sW$.
In addition, since $\sx{N+1}$ is a minimizer of $\sW$, we get that $Q^T\sx{N+1}\in X_*(w_\algorithm)$.
Finally, from Lemma~\ref{L:xstardist} and Corollary~\ref{C:minvalue} we have
\begin{align*}
    & \|Q^T \sx{N+1}- x_0\|=\|\sx{N+1}\|= R, \\
    \text{and }\quad & w_\algorithm(\ax_N) - w_\algorithm^*= w_\algorithm(\ax_N)\geq \sff{N} = \frac{L R^2}{2\theta_N^2},
\end{align*}
hence $w_\algorithm$ satisfies the required properties.
\end{proof}

Combining the previous theorem and the worst-case performance of the Optimized Gradient Method~\eqref{E:upperboundonrisk}, we obtain
the exact minimax risk associated with the minimization of smooth and convex functions:
\begin{corollary}
Let $L>0$, $R>0$,  and $N\in \mathbb{N}$, then for any $d\geq N+1$,
\[
    \risk{\finstances}{N} =  \frac{LR^2}{2\theta_N^2},
\]
where $\theta_N$ is as defined in \eqref{E:thetadef}.
\end{corollary}

\begin{remark}\label{R:nonsmoothex}
Note that the proof of Theorem~\ref{T:lowerbound} can be readily applied on any function satisfying 
Corollaries~\ref{C:minvalue} and~\ref{C:zerodrivative} (with the obvious adjustments).
One important such function
is the function $\bar f_N$ defined in \cite[Appendix~A]{drori2014optimal}
for establishing the minimax risk for the non-smooth case: since this function clearly satisfies the required properties,
it then follows by Theorem~\ref{T:lowerbound} that the identity~\eqref{E:nonsmoothbound} holds for $d\geq N+1$ 
(compared to $d\geq 2N+1$ which follows by the current proof).
\end{remark}

\section{Discussion an future work}\label{S:concluding}
We presented a novel interpolation scheme for constructing smooth and convex functions
and used it to define a function that is in some sense the worst possible for first-order methods.
Taking advantage of the special properties of the resulting function,
we established the exact minimax risk associated with the class of smooth and convex minimization problems.

A similar approach might be applicable for establishing  bounds on additional problem classes,
such as classes of problems where the approximation accuracy is measured by an alternative criteria (e.g., $\|\nabla f(x_N)\|$)
and cases where the relation between $x_0$ and $x_*$ is determined in an alternative way (e.g., when $\|x_0-x_*\|_p\leq R$ for some $p\neq 2$).
We leave the investigation of these problems for future work.

We conclude with the observation that, as in the cases of non-smooth minimization and convex quadratic minimization, 
an optimal method for smooth minimization might not be unique.
Finding optimal methods with additional properties might prove to be an interesting research direction.

\bibliographystyle{abbrv}
\bibliography{bib}

\begin{thebibliography}{10}

\bibitem{agarwal2012information}
A.~Agarwal, P.~L. Bartlett, P.~Ravikumar, and M.~J. Wainwright.
\newblock Information-theoretic lower bounds on the oracle complexity of
  stochastic convex optimization.
\newblock {\em IEEE Trans. Inf. Theory}, 5(58):3235--3249, 2012.

\bibitem{bertsekas1999nonlinear}
D.~P. Bertsekas.
\newblock {\em Nonlinear programming}.
\newblock Athena scientific, 1999.

\bibitem{drori2013performance}
Y.~Drori and M.~Teboulle.
\newblock Performance of first-order methods for smooth convex minimization: a
  novel approach.
\newblock {\em Math. Program. Ser. A}, 145:451--482, 2014.

\bibitem{drori2014optimal}
Y.~Drori and M.~Teboulle.
\newblock An optimal variant of kelley's cutting-plane method.
\newblock {\em Math. Program. Ser. A}, pages 1--31, 2016.

\bibitem{Guzman20151}
C.~Guzmán and A.~Nemirovski.
\newblock On lower complexity bounds for large-scale smooth convex
  optimization.
\newblock {\em J. Complexity}, 31(1):1 -- 14, 2015.

\bibitem{kim2015optimized}
D.~Kim and J.~Fessler.
\newblock Optimized first-order methods for smooth convex minimization.
\newblock {\em Math. Program. Ser. A}, pages 1--27, 2015.

\bibitem{nemirovski1999optimization}
A.~Nemirovski.
\newblock Optimization {II}: Numerical methods for nonlinear continuous
  optimization.
\newblock {\em Lecture notes,
  \url{http://www2.isye.gatech.edu/~nemirovs/Lect_OptII.pdf}}, 1999.

\bibitem{nemirovsky1992information}
A.~Nemirovsky.
\newblock Information-based complexity of linear operator equations.
\newblock {\em J. Complexity}, 8(2):153--175, 1992.

\bibitem{nemi-yudi-book83}
A.~S. Nemirovsky and D.~B. Yudin.
\newblock {\em Problem complexity and method efficiency in optimization}.
\newblock A Wiley-Interscience Publication. John Wiley \& Sons Inc., New York,
  1983.
\newblock Translated from the Russian and with a preface by E. R. Dawson,
  Wiley-Interscience Series in Discrete Mathematics.

\bibitem{nest-book-04}
Y.~Nesterov.
\newblock {\em {Introductory lectures on convex optimization: a basic course}}.
\newblock Applied optimization. Kluwer Academic Publishers, 2004.

\bibitem{raginsky2011information}
M.~Raginsky and A.~Rakhlin.
\newblock Information-based complexity, feedback and dynamics in convex
  programming.
\newblock {\em Information Theory, IEEE Transactions on}, 57(10):7036--7056,
  2011.

\bibitem{rockafellar2009variational}
R.~T. Rockafellar and R.~J.-B. Wets.
\newblock {\em Variational analysis}, volume 317.
\newblock Springer Science \& Business Media, 2009.

\bibitem{shapiro2005complexity}
A.~Shapiro and A.~Nemirovski.
\newblock On complexity of stochastic programming problems.
\newblock In {\em Continuous optimization}, pages 111--146. Springer, 2005.

\bibitem{Taylor2016}
A.~B. Taylor, J.~M. Hendrickx, and F.~Glineur.
\newblock Smooth strongly convex interpolation and exact worst-case performance
  of first-order methods.
\newblock {\em Math. Program. Ser. A}, pages 1--39, 2016.

\end{thebibliography}

\end{document}